\documentclass[11pt, a4paper]{article}
\usepackage{amsmath,amssymb,amsthm}
\usepackage{graphicx,xcolor}
\usepackage{subfig}

\usepackage[auth-sc, affil-it]{authblk}

\usepackage[utf8]{inputenc}
\usepackage[english]{babel}
\usepackage[T1]{fontenc}
\usepackage{lmodern}
\usepackage{microtype}
\usepackage[
    linktocpage=true,
    colorlinks=true,
    linkcolor=purple,
    citecolor=black,
    urlcolor=purple,
    bookmarksnumbered=true,
]{hyperref}

\usepackage[totalwidth=160mm, totalheight=247mm]{geometry}

\usepackage{enumitem}
\setlist{noitemsep}
\setlist[enumerate]{label=(\arabic*)}

\let\oldbibliography\thebibliography
\renewcommand{\thebibliography}[1]{%
  \oldbibliography{}%
  \small%
  \setlength{\itemsep}{0pt}%
}

\makeatletter
\def\@cite#1#2{[{{\bfseries#1}\if@tempswa , #2\fi}]} 
\renewcommand{\@biblabel}[1]{[{\bfseries{#1}}]~} 
\makeatother

\usepackage[format=plain,justification=justified]{caption}
\theoremstyle{plain}
\newtheorem{theo}{Theorem}[section]
\newtheorem{prop}[theo]{Proposition}
\newtheorem{lemm}[theo]{Lemma}

\theoremstyle{definition}
\newtheorem{defi}[theo]{Definition}
\newtheorem{exam}[theo]{Example}
\newtheorem{rema}[theo]{Remark}

\renewcommand{\leq}{\leqslant}
\renewcommand{\geq}{\geqslant}
\newcommand{\bbN}{\mathbb{N}}
\newcommand{\bbZ}{\mathbb{Z}}

\newcommand{\bbR}{\mathbb{R}}
\newcommand{\bbC}{\mathbb{C}}
\newcommand{\mcA}{\mathcal{A}}
\newcommand{\mcT}{\mathcal{T}}
\newcommand{\bfM}{\mathbf{M}}
\newcommand{\bfe}{\mathbf{e}}
\newcommand{\bfP}{\mathbf{P}}
\newcommand{\bfh}{\mathbf{h}}

\newcommand{\Ts}{\mcT_\sigma}
\newcommand{\Tt}{\mcT_\tau}
\newcommand{\Tth}{\mcT_\theta}
\newcommand{\Ms}{\mathbf{M}_\sigma}
\newcommand{\Mt}{\mathbf{M}_\tau}
\newcommand{\Mth}{\mathbf{M}_\theta}

\newcommand{\bfv}{\mathbf{v}}
\newcommand{\bfw}{\mathbf{w}}

\newcommand{\vbeta}{\mathbf{v}_\beta}
\newcommand{\wbeta}{\mathbf{w}_\beta}
\newcommand{\zbeta}{\mathbf{z}_\beta}
\DeclareMathOperator{\occ}{occ}
\DeclareMathOperator{\diag}{diag}
\DeclareTextFontCommand{\tdef}{\itshape\bfseries}
\newcommand{\myvcenter}[1]{\ensuremath{\vcenter{\hbox{#1}}}}

\title{\textbf{Rauzy fractals with countable fundamental group}}
\author[1,2]{Timo Jolivet}
\author[3]{Beno\^it Loridant}
\author[4]{Jun Luo}

\affil[1]{%
    Universit\'e Paris Diderot,
    LIAFA Case 7014,
    75205 Paris Cedex 13, France
}
\affil[2]{%
    Department of Mathematics,
    University of Turku 20014, Finland
}
\affil[3]{%
    Montanuniversit\"at Leoben,
    Franz Josefstrasse 18, Leoben 8700, Austria
}
\affil[4]{%
    Department of Statistics,
    Sun Yat-Sen University, Guangzhou 512075, China
}

\date{}

\begin{document}
\maketitle

\begin{abstract}
We prove that every free group of finite rank can be realized
as the fundamental group of a planar Rauzy fractal associated
with a $4$-letter unimodular cubic Pisot substitution.
This characterizes all countable fundamental groups for planar Rauzy fractals.
We give an explicit construction relying on two operations on substitutions:
symbolic splittings and conjugations by free group automorphisms.
\end{abstract}

\section{Introduction}
In 1982, Rauzy proved that the dynamical system generated by
the Tribonacci substitution $\sigma(1)=12, \sigma(2)=13, \sigma(3)=1$
is measure theoretically conjugate to an exchange of domains on a compact subset of
the plane with fractal boundary~\cite{Rau82}.
He even showed that this dynamical system is measure theoretically conjugate to
a translation on the two-dimensional torus: in other words, it has pure discrete spectrum.
These results were later generalized to every primitive irreducible unimodular Pisot substitution
satisfying combinatorial conditions called \emph{coincidence conditions}~\cite{AI01,CS01,IR06,BBK06}.
The Pisot conjecture states that such systems always have pure discrete spectrum~\cite{ABBLS14}.

The associated \emph{Rauzy fractals} and their subdomains are compact sets equal to the closure of their interior,
and they are attractors of graph directed iterated function systems~\cite{SW02,ST09}.
Besides these common properties, Rauzy fractals enjoy a great topological diversity.
In the literature, properties like connectedness,
homeomorphy to a closed disc for planar Rauzy fractals or
triviality of their fundamental group are investigated.
Most of these questions can be solved algorithmically for a given substitution~\cite{ST09}.

The study of Rauzy fractals and their topological properties is motivated by several applications.
Examples are the numeration systems with non-integer bases (see~\cite{Thu89} and the survey~\cite{BS05});
the computation of simultaneous Diophantine approximations~\cite{HM06};
the theory of tiling dynamical systems~\cite{S97,BBK06};
the generation of discrete planes related to multidimensional continued fraction algorithms~\cite{IO94,ABI02};
the relation with some topological invariants of tiling spaces~\cite{BDS09};
and the search for explicit Markov partitions for hyperbolic toral automorphisms~\cite{IO93,KV98,Adl98,Pra99}.

In the planar case, there are many known examples of Rauzy fractals which are homeomorphic to a disc,
or whose fundamental group is uncountable~\cite{Mes98,Mes06,ST09,LMST13}.
However, until now, no example with ``intermediate'' constellation is known,
where the fundamental group would be nontrivial, but countable.

In this article, we prove that such an intermediate situation occurs
by giving a method to construct explicit examples.
For any given $K \in \bbN$, we are able to construct a $4$-letter primitive unimodular cubic Pisot substitution
whose Rauzy fractal has a fundamental group isomorphic to the free group $F_K$ of rank $K$
(Theorem~\ref{theo:countablefg}).
This result is complete in the sense that every countable fundamental group of a planar Rauzy fractal
must be of this form (Proposition~\ref{prop:cl_she}).

Our method relies on two symbolic operations on substitutions
that induce manipulations on the subtiles of the associated Rauzy fractals,
namely \emph{symbol splittings} and \emph{conjugation by free group automorphisms}.
Questions about the effect of conjugation by free group automorphisms on Rauzy fractals
have already been raised in~\cite{Gah10} and~\cite{ABHS06}.
A consequence of our work is that
the fundamental group of the Rauzy fractal of a substitution $\sigma$
is not preserved after conjugation of $\sigma$  by free group automorphisms.

\paragraph{Outline}

The paper is organized as follows.
In Section~\ref{sect:prelim}, together with preliminary results, we recall that
Rauzy fractals can naturally be decomposed into \emph{subtiles} and \emph{subsubtiles}.
We then manipulate these tile subdivisions within the fractal in order to obtain the desired topological properties.
Our tools consist of two symbolic operations on substitutions:
\emph{symbol splittings} (Section~\ref{sect:split})
and \emph{conjugation by free group automorphisms} (Section~\ref{sect:fgaut}).
Our main results are proved in Section~\ref{sect:mainresults}.
Schematically, they are obtained via the following strategy (see Figure~\ref{fig:strategy}):
\begin{itemize}
\item[\subref{fig:strategy-a}] Start with a substitution $\sigma$ on three symbols
    whose Rauzy fractal and its subtiles are disklike.
\item[\subref{fig:strategy-b}] Take $n$ large enough such that the subtiles of $\sigma^n$
    consist of sufficiently small subsubtiles for the next two steps to be applicable (Proposition~\ref{prop:itersubtiles}).
\item[\subref{fig:strategy-c}] Split a symbol to isolate a subsubtile and
    turn it into a subtile of the Rauzy fractal of a new substitution $\tau$
    on four symbols (Proposition~\ref{prop:split}).
\item[\subref{fig:strategy-d}] Conjugate $\tau$ by a suitable free group automorphism $\rho$.
    The Rauzy fractal associated with $\rho^{-1}\tau\rho$ now has a hole (Proposition~\ref{prop:fgaut}).
\end{itemize}
\begin{figure}[ht]%
\centering
\subfloat[][Subtiles of $\sigma$]{%
    \label{fig:strategy-a}%
    \includegraphics[height=35mm]{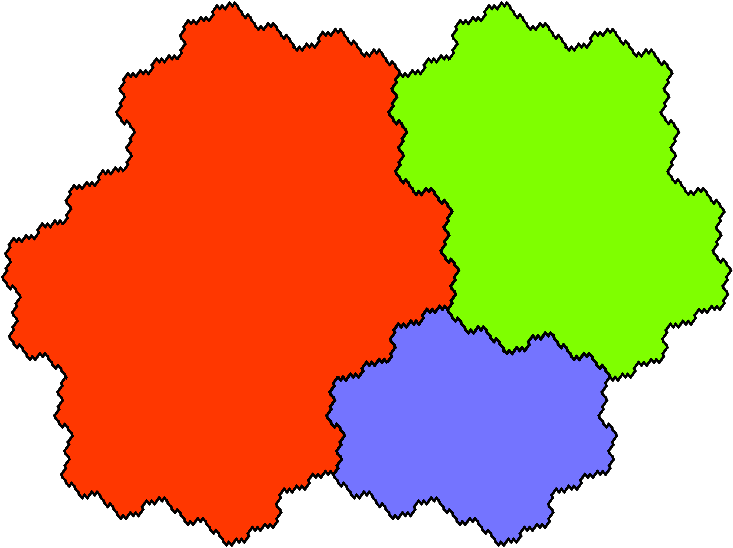}
} \hfil
\subfloat[][Subsubtiles of $\sigma^3$]{%
    \label{fig:strategy-b}%
    \includegraphics[height=35mm]{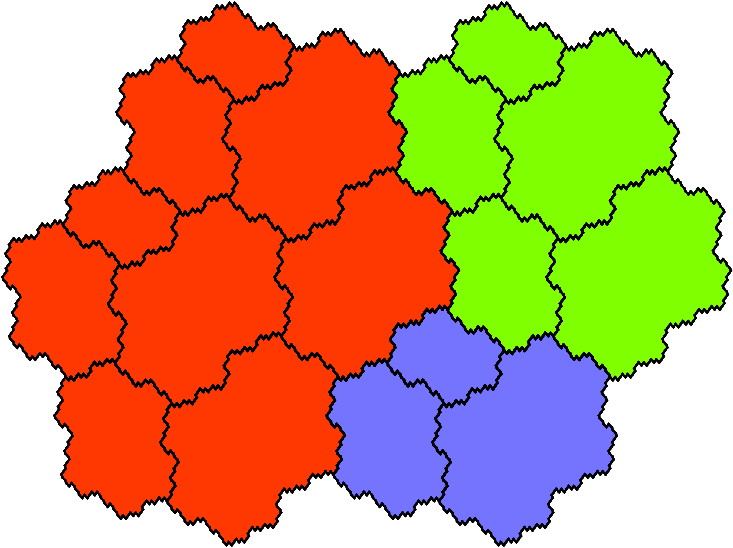}
}\\
\subfloat[][Subtiles after splitting a symbol in $\sigma^3$]{%
    \label{fig:strategy-c}%
    \includegraphics[height=35mm]{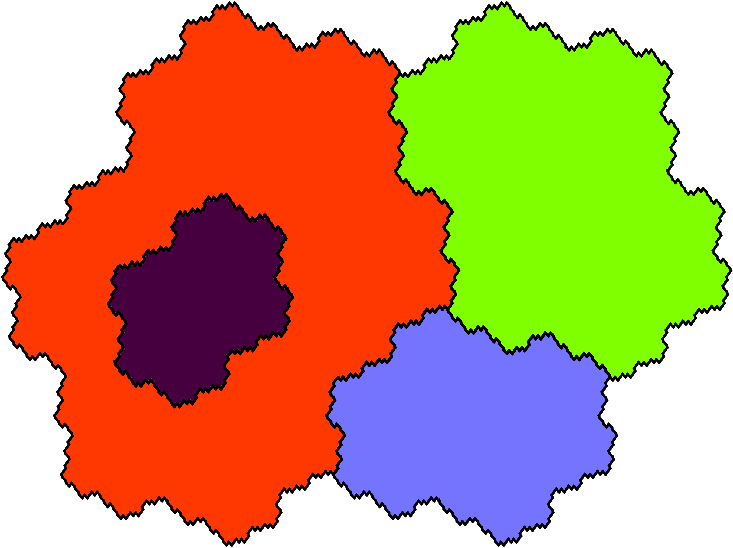}
} \hfil
\subfloat[][Subtiles after conjugating by a free group automorphism]{%
    \label{fig:strategy-d}%
    \includegraphics[height=35mm]{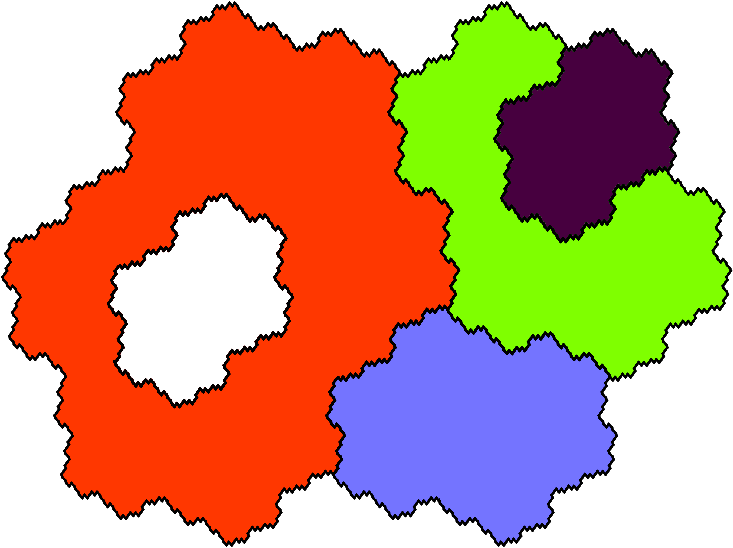}
}%
\caption[]{The main steps of our strategy.}%
\label{fig:strategy}%
\end{figure}

Note that the subtiles in Figure~\ref{fig:strategy}~\subref{fig:strategy-d} do not overlap (see Theorem~\ref{theo:countablefg}).
The fractals in Figure~\ref{fig:strategy}~\subref{fig:strategy-c} and~\ref{fig:strategy}~\subref{fig:strategy-d}
have different areas, this is explained in Remark~\ref{rema:crucial}.

\section{Preliminaries}
\label{sect:prelim}

In the following, $\mcA$ denotes a finite set of symbols,
and $\mcA^\star$ denotes the \tdef{free monoid over $\mcA$}
defined as the set of all finite words over $\mcA$,
where the composition of two words $u$ and $v$ is their concatenation $uv$.
If $w$ is an element of $\mcA^\star$ or $\mcA^\bbN$,
its $i$-th letter is denoted by~$w_i$.

\subsection{Substitutions}

Let $\mcA$ be a finite set of symbols.
A \tdef{substitution} is a non-erasing morphism of the free monoid~$\mcA^\star$,
\emph{i.e.}, a function $\sigma : \mcA^\star \rightarrow \mcA^\star$
such that $\sigma(uv) = \sigma(u)\sigma(v)$ for all words $u, v \in \mcA^\star$,
and such that $\sigma(a)$ is nonempty for every $a \in \mcA$.

We denote by $\bfP : \mcA^\star \rightarrow \bbZ^n$
the \tdef{Abelianization map} defined by
$\bfP(w) = (|w|_1, \ldots ,|w|_n)$,
where $|w|_i$ denotes the number of occurrences of $i$ in $w$.
The \tdef{incidence matrix} $\Ms$ of $\sigma$ is the matrix of size $n \times n$
whose $i$-th column is equal to $\bfP(\sigma(i))$ for every $i \in \mcA$.
A substitution $\sigma$ is
\begin{itemize}
\item \tdef{unimodular} if $\det(\Ms) = \pm 1$;
\item \tdef{primitive} if $\Ms$ is primitive
(all the entries of $\Ms^n$ are strictly positive for some $n \geq 1$);
\item \tdef{Pisot} if the dominant eigenvalue of $\Ms$ is a Pisot number: an algebraic integer $\beta>1$
    whose Galois conjugates $\beta_1,\ldots,\beta_d$ satisfy $\left|\beta_i\right| < 1$;
\item \tdef{irreducible} if the algebraic degree $d$ of the dominant eigenvalue $\beta$ of $\Ms$
    is equal to the size of the alphabet of $\sigma$.
\end{itemize}

An infinite word $u \in \mcA^\bbN$ is a \tdef{periodic point} of $\sigma$
if there exists $k \in \bbN$ such that $\sigma^k(u) = u$.
Such a periodic point always exists when $\sigma$ is primitive~\cite[Proposition~5.1]{Que10}.

\subsection{Rauzy fractals and subtiles}

Before defining Rauzy fractals we introduce the necessary algebraic setup.
Let $\sigma$ be a primitive unimodular Pisot substitution on the alphabet $\mcA = \{1, \ldots, n\}$,
and let $\beta$ be the Pisot the dominant real eigenvalue of $\bfM_\sigma$, a Pisot number of degree $d$.
Denote by $\beta_1, \ldots, \beta_r$ the $r$ real conjugates of $\beta$,
and denote by $\beta_{r+1}, \overline{\beta_{r+1}}, \ldots, \beta_{r+s}, \overline{\beta_{r+s}}$
the $2s$ complex conjugates of $\beta$ (we have $r + 2s = d-1$).
Let $\vbeta$ be a left eigenvector of $\bfM_\sigma$ associated with $\beta$.
Let $\pi_\sigma$ the projection given by
\[
\begin{array}{cccll}
\pi_\sigma & : &
\bbR^n
& \rightarrow &
\bbR^r \times \bbC^s \cong \bbR^{d-1} \\
& & \bfe_i
& \mapsto &
(
\langle \bfv_{\beta_1}, \bfe_i \rangle,
\ldots,
\langle \bfv_{\beta_{r+s}}, \bfe_i \rangle
)
\end{array}
\]
where each eigenvector $\bfv_{\beta_j}$ is obtained by replacing $\beta$ by $\beta_j$ in the coordinates of $\vbeta$.
Note that the conjugates $\smash{\overline{\beta_{r+1}}, \ldots, \overline{\beta_{r+s}}}$
are not taken into account in the definition of $\pi_\sigma$.

\begin{defi}
\label{defi:rf}
Let $\sigma$ be a primitive unimodular Pisot substitution on the alphabet $\mcA$
and let $u$ be a periodic point of $\sigma$.
The \tdef{Rauzy fractal} of $\sigma$ (with respect to $\vbeta$)
is the set $\Ts = \bigcup_{i \in \mcA} \Ts(i)$,
where for each $i \in \{1,\ldots,n\}$, $\Ts(i)$ is the \tdef{subtile} of type $i$ given by
\[
\Ts(i) = \overline{\{\pi_\sigma \bfP(u_1 \ldots u_m) : m \in \bbN \text{ and } u_{m+1} = i\}}.
\]
\end{defi}

\begin{rema}
\label{rema:crucial}
In the above definition,
the norm of $\vbeta$ \emph{does} affect the area of the sets $\Ts$ and $\Ts(i)$ up to an inflation factor.
Standard definitions of Rauzy fractals usually require $\|\vbeta\|_1 = 1$.
In this article, we will not put any restriction on the norm of $\vbeta$,
always specifying with respect to which $\vbeta$ we define a Rauzy fractal.
This will help us to avoid many technical difficulties
when relating different Rauzy fractals (living in different representation spaces)
in Proposition~\ref{prop:split}~and Proposition~\ref{prop:fgaut}.
\end{rema}

\subsection{Subsubtiles and graph-directed iterated function system}
\label{sec:subsub}

In the definitions below we will need the mapping
$\bfh_\sigma : \bbR^r \times \bbC^s \rightarrow \bbR^r \times \bbC^s$, defined by
$\bfh_\sigma(x) = \diag(\beta_1, \ldots, \beta_{r+s}) x$.
The mapping $\bfh_\sigma$ is contracting on $\bbR^r \times \bbC^s$
because $|\beta_i| < 1$ for $1 \leq i \leq r+s$.
It corresponds to the action of $\bfM_\sigma$ before projecting by $\pi_\sigma$,
in other words, $\pi_\sigma\Ms = \bfh_\sigma\pi_\sigma$.

In Definition~\ref{defi:rf}, we have given a decomposition of the tile $\Ts$ into its subtiles $\Ts(i)$.
In Sections~\ref{sect:split} and~\ref{sect:fgaut} we will need to decompose Rauzy fractals one step further:
each subtile $\Ts(i)$ can be decomposed into its \emph{subsubtiles} $\Ts(i,j;k)$,
defined below in Definition~\ref{defi:subsubtiles}.

Intuitively, each subsubtile of $\Ts(i)$ corresponds to an occurrence of $i$ in the words $\sigma(j)$.
We formalize the notion of occurrence before defining subsubtiles.
A pair $(j,k) \in \mcA \times \bbN$ is an \tdef{occurrence} of the symbol $i$ in $\sigma$ if
$\sigma(j)_k = i$, that is if the $k$-th letter of $\sigma(j)$ is $i$.
We will denote occurrences by $(j;k)$ to emphasize the fact that $j$ is an element of $\mcA$
and $k$ is an index.
The \tdef{set of occurrences} of $i$ in $\sigma$ is denoted by $\occ(\sigma,i)$.
For example, for $\sigma : 1 \mapsto 11213, 2 \mapsto 331, 3 \mapsto 1$ we have
$\occ(\sigma,1) = \{(1;1), (1;2), (1;4), (2;3), (3;1)\}$,
$\occ(\sigma,2) = \{(1;3)\}$ and
$\occ(\sigma,3) = \{(1;5), (2;1), (2;2)\}$.

\begin{defi}
\label{defi:subsubtiles}
Let $(j;k) \in \occ(\sigma,i)$.
The \tdef{subsubtile} $\Ts(i,j;k)$ is defined by
\[
\Ts(i,j;k) = \bfh_\sigma(\Ts(j)) + \pi_\sigma\bfP(\sigma(j)_1\cdots\sigma(j)_{k-1}).
\]
\end{defi}

Note that $\Ts(i,j;k)$ is defined only if $(j;k) \in \occ(\sigma,i)$.
The tiles $\Ts(i)$ are the solution of a graph-directed iterated function system,
which can be conveniently expressed in terms of subsubtiles and symbol occurrences in the following theorem.

\begin{prop}[\cite{SW02,EIR06}]
\label{prop:GIFS}
Let $\sigma$ be primitive unimodular Pisot substitution on the alphabet $\mcA$.
For every $i \in \mcA$ we have
\[\Ts(i) =
    \bigcup_{(j;k) \in \occ(\sigma,i)}
    \mcT_{\sigma}(i,j;k),
\]
and this union is measure-disjoint.
\end{prop}

The proof for the measure-disjointness is given in~\cite{SW02} for the irreducible case
and in~\cite{EIR06} for the reducible case. We also refer to~\cite{AI01,BS05,ST09,CANT}.

\begin{exam}
\label{exam:subsubtiles}
Let $\sigma : 1 \mapsto 21, 2 \mapsto 31, 3 \mapsto 1$.
The subsubtiles of $\sigma^3 : 1 \mapsto 1213121, 2 \mapsto 213121, 3 \mapsto 3121$
are plotted in Figure~\ref{fig:strategy}~\subref{fig:strategy-b}.
The~$9$ subsubtiles of $ \mcT_{\sigma^3}(1)=\Ts(1)$ correspond to the $9$ occurrences of $1$ in $\sigma^3$;
the~$5$ subsubtiles of $\mcT_{\sigma^3}(2)$ correspond to the $5$ occurrences of $2$ in $\sigma^3$;
the~$3$ subsubtiles of $\mcT_{\sigma^3}(3)$ correspond to the $3$ occurrences of $3$ in $\sigma^3$.
\end{exam}

\begin{rema}
\label{rema:SCC}
According to see~\cite{AI01}, the subtiles $\Ts(i)$, $i\in\mcA$, are measure-disjoint
if $\sigma$ satisfies the \tdef{strong coincidence condition}:
for every $(j_1,j_2) \in \mcA^2$,
there exists $k \in \bbN$ and $i \in \mcA$
such that $\sigma^k(j_1) = p_1 i s_1$ and $\sigma^k(j_2) = p_2 i s_2$
with $\bfP(p_1) = \bfP(p_2)$ or $\bfP(s_1) = \bfP(s_2)$.
\end{rema}

\begin{rema}
\label{rema:powers}
We mention that for $N\geq 1$ and $i \in \mcA$
we have $\Ts(i)=\mcT_{\sigma^N}(i)$.
Moreover, iterating the above equation, we obtain for all $N\geq 1$ and for every $i \in \mcA$ we have
\[
\Ts(i) =
    \bigcup_{(j;k) \in \occ(\sigma^N,i)}
    \mcT_{\sigma^N}(i,j;k).
\]
\end{rema}

\subsection{Countable fundamental groups of Rauzy fractals are free}

We now prove that free groups of finite rank are the only possible countable fundamental groups of Rauzy fractals.
Let us recall the following basic notions and results from topology~\cite{WhyDuda79}.
A topological space $X$ is a \tdef{continuum} if it is compact and connected.
It is \tdef{locally connected} if it has a base of connected sets.
A \tdef{path from $x$ to} $y$ in $X$ is a continuous function $f:[0,1]\to X$ with $f(x)=0$ and $f(y)=1$.
$X$ is \tdef{path-connected} if every two points of $X$ are joined by a path,
and \tdef{locally path-connected} if it has a base of path-connected sets.
It follows from the theorem of Hahn-Mazurkiewicz that any locally connected continuum is path-connected.
Moreover, in a metric space, every locally connected continuum is locally path-connected
by results of Mazurkiewicz, Moore and Menger (see~\cite[Section~50, Chapter~II, p.~254]{Kur68}).

\begin{prop}
\label{prop:cl_she}
Let $\sigma$ be a primitive unimodular Pisot substitution
and let $\Ts$ be its Rauzy fractal.
Suppose that $\Ts$ and its subtiles are planar locally connected continua.
If the fundamental group of $\Ts$ is countable,
then it is isomorphic to the free group $F_K$ on $K$ generators for some finite rank~$K$.
\end{prop}

\begin{proof}
The result follows directly from a theorem of Conner and Lamoreaux~\cite[Theorem~3.1]{CL05},
which states that if a planar set is connected and locally path-connected,
then its fundamental group is not free if and only if it is uncountable.
Note that this result can also be proved using a theorem of Shelah~\cite{She88}.
\end{proof}

\section{Symbol splittings}
\label{sect:split}

We now define a symbolic operation, \emph{symbol splitting},
that we will use in Proposition~\ref{prop:split}.

\begin{defi}
Let $\sigma$ be a substitution on the alphabet $\mcA$,
let $a \in \mcA$,
let $b \notin \mcA$ be a new symbol
and let $I \subseteq \occ(\sigma,a)$ be a nonempty set of occurrences of $a$ in $\sigma$.
The \tdef{splitting} of symbol $a$ to the new symbol $b$ with occurrences $I$
is the substitution $\tau$ defined by
\[
\tau(i) =
\begin{cases}
\text{the word $\sigma(i)$ in which $\sigma(i)_k$ is replaced by $b$ for every $(i;k) \in I$, if  $i \neq b$,} \\
\tau(a), \text{ if } i = b.
\end{cases}
\]
\end{defi}

See Example~\ref{exam:split} for an example of symbol splitting.
Note that if $\sigma$ is a primitive unimodular Pisot substitution,
then so also is any splitting $\tau$ arising from $\sigma$ (Lemma~\ref{lemm:split_lemm1}).
Moreover, we have $\chi_\tau(x) = x\cdot\chi_\sigma(x)$,
where $\chi_\sigma$ and $\chi_\tau$ are the characteristic polynomials
of $\bfM_\sigma$ and $\bfM_\tau$, respectively.
The action of symbol splittings on the Rauzy fractal of a substitution $\sigma$ is described in the next proposition,
and is illustrated in Example~\ref{exam:split}.

\begin{prop}
\label{prop:split}
Let
\begin{itemize}
\item
$\sigma$ be a primitive unimodular Pisot substitution on the alphabet $\mcA = \{1, \ldots, n\}$,
\item
$\tau$ be obtained by splitting of $\sigma$ from symbol $a$ to a new symbol $b = n+1$ with occurrences $I\subseteq \occ(\sigma,a)$,
\item
$\vbeta = (v_1, \ldots, v_n)\in\bbR^n$ be a left eigenvector of $\Ms$ associated with $\beta$,
\item
$\wbeta = (v_1, \ldots, v_n, v_a)\in\bbR^{n+1}$
    (which is a left eigenvector of $\Mt$ associated with $\beta$, see Lemma~\ref{lemm:split_lemm1}),
\item
$\Ts$ be the Rauzy fractal of $\sigma$ (with respect to the eigenvector $\vbeta$),
\item
$\Tt$ be the Rauzy fractal of $\tau$ (with respect to the eigenvector $\wbeta$).
\end{itemize}
We have
\begin{enumerate}[itemsep=2pt]
\item
\label{item:split1}
$\Tt(i) = \Ts(i)$ if $i \notin \{a,b\}$,$\phantom{\displaystyle\bigcup_{(}}$
\item
\label{item:split2}
$\Tt(a) = \displaystyle\bigcup_{(j;k) \in \occ(\sigma,a) \setminus I} \Ts(a,j;k)$,
\item
\label{item:split3}
$\Tt(b) = \displaystyle\bigcup_{(j;k) \in I} \Ts(a,j;k)$.
\end{enumerate}
\end{prop}

\begin{exam}
\label{exam:split}
Let $\sigma : 1 \mapsto 1213121, 2 \mapsto 213121, 3 \mapsto 3121$.
We split the symbol $a = 1$ to the new symbol $b = 4$
with occurrences $I = \{(1;1), (2;6), (3;2)\}$ of $1$ in $\sigma$.
The resulting substitution $\tau$ and its Rauzy fractal are shown below.
(The tiles associated with $4$ are shown in black.)
\[
\tau :
\begin{cases}
1 \mapsto 4213121 \\
2 \mapsto 213124 \\
3 \mapsto 3421 \\
4 \mapsto 4213121
\end{cases}
\qquad
\qquad
\myvcenter{\includegraphics[height=3cm]{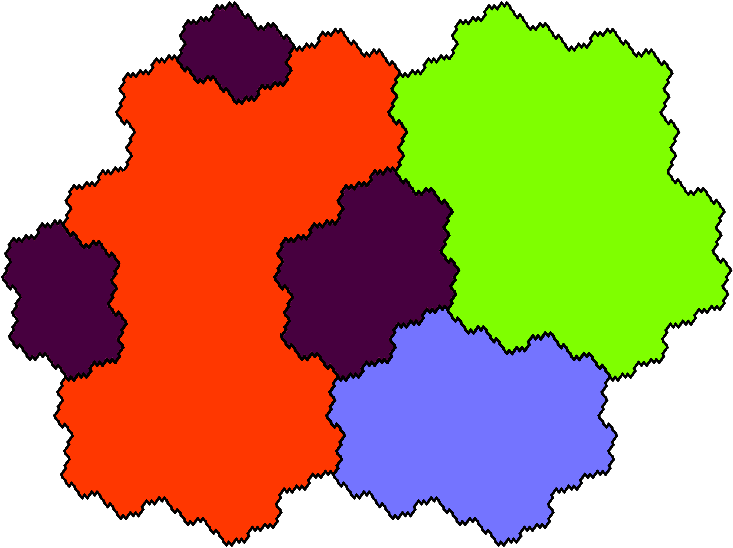}}
\]
\end{exam}

In order to prove Proposition~\ref{prop:split}
we need Lemma~\ref{lemm:split_lemm1} and Lemma~\ref{lemm:split_lemm2} below.

\begin{lemm}
\label{lemm:split_lemm1}
Under the hypotheses of Proposition~\ref{prop:split},
$\beta$ is an eigenvalue of $\Mt$
and $\wbeta$ a left eigenvector of $\Mt$ associated with $\beta$.
Hence, the Rauzy fractal $\Tt$ mentioned in the statement of Proposition~\ref{prop:split} is well-defined.
\end{lemm}

\begin{proof}
By definition of symbol splittings we have
\begin{itemize}
\item $(\Mt)_{i,j} = (\Ms)_{i,j}$ for all $i \notin \{a, n+1\}$ and $j \neq n+1$,
\item $(\Mt)_{n+1,j} = (\Ms)_{a,j}$ for all $j \notin \{a, n+1\}$,
\item $(\Mt)_{i,a} + (\Mt)_{i,n+1} = (\Ms)_{i,a}$ for all $i\neq n+1$.
\end{itemize}
Hence, by definition of $\wbeta$ we have $(\wbeta \Mt)_i = (\vbeta \Ms)_i$ if $i \neq n+1$
and $(\wbeta \Mt)_{n+1} = (\vbeta \Ms)_a$,
so $\wbeta \Mt = \beta \wbeta$,
which proves the lemma.
\end{proof}

\begin{lemm}
\label{lemm:split_lemm2}
Under the hypotheses of Proposition~\ref{prop:split},
let $i\in\{1,\ldots,n,b=n+1\}$ and $(j;k) \in \occ(\tau,i)$,
and let $i' = a$ if $i = b$ and $i' = i$ otherwise.
We have
$\Tt(i,j;k) = \Ts(i',j;k)$ if $j \notin \{a,b\}$,
and
$\Tt(i,a;k) \cup \Tt(i,b;k) = \Ts(i',a;k)$.
\end{lemm}

\begin{proof}
First, note that $\bfh_\sigma = \bfh_\tau$ by definition,
they are both equal to $\diag(\beta_1, \ldots, \beta_{r+s})$
(recall that $\chi_\tau(x) = x\cdot\chi_\sigma(x)$, see Section~\ref{sec:subsub}).
Also, by Lemma~\ref{lemm:split_lemm1}, for $j=1,\ldots,r+s$, $\beta_j$ is an eigenvalue
and $\bfw_{\beta_j}$ is a left eigenvector of $\Mt$ associated with $\beta_j$,
where $\bfw_{\beta_j}$ is obtained by replacing $\beta$ by $\beta_j$ in the coordinates of $\wbeta$.
It follows that $\pi_\tau\bfP(i) = \pi_\sigma\bfP(i)$ for all $i \neq b$
and $\pi_\tau\bfP(b) = \pi_\sigma\bfP(a)$.
We will use these facts later in the proof.

Next, we claim that $\Tt(i) = \Ts(i)$ if $i \notin \{a,b\}$ and $\Tt(a) \cup \Tt(b) = \Ts(a)$.
Indeed, let $u$ be a periodic point of $\tau$,
and let $u'$ be defined by $u_m' = a$ if $u_m = b$ and $u_m' = u_m$ otherwise.
Then it is easy to check that $u'$ is a periodic point of $\sigma$,
and that $\pi_\tau \bfP(u_1\cdots u_m) = \pi_\sigma \bfP(u_1' \cdots u_m')$ for all $m \geq 1$,
so our claim follows from Definition~\ref{defi:rf} of Rauzy fractals.
Finally, we have
\begin{align*}
\Tt(i,j;k)
    &= \bfh_\tau \Tt(j) + \pi_\tau \bfP(\tau(j)_1 \cdots \tau(j)_{k-1}) \\
    &= \bfh_\sigma \Ts(j) + \pi_\sigma \bfP(\sigma(j)_1 \cdots \sigma(j)_{k-1}) \\
    &= \Ts(i',j;k) \textrm{ for } j\notin\{a,b\},
\end{align*}
and
\begin{align*}
\Tt(i,a;k) \cup \Tt(i,b;k)
    &= \bfh_\tau (\Tt(a) \cup \Tt(b)) + \pi_\tau \bfP(\tau(a)_1 \cdots \tau(a)_{k-1}) \\
    &= \bfh_\tau \Ts(a) + \pi_\tau \bfP(\tau(a)_1 \cdots \tau(a)_{k-1}) \\
    &= \bfh_\sigma \Ts(a) + \pi_\sigma \bfP(\sigma(a)_1 \cdots \sigma(a)_{k-1}) \\
    &= \Ts(i',a;k),
\end{align*}
which proves the lemma.
\end{proof}

\begin{proof}[Proof of Proposition~\ref{prop:split}]
Let $(j;k) \in \occ(\tau,i)$,
and let $i' = a$ if $i = b$ and $i' = i$ otherwise.
We have
\begin{align*}
\Tt(i)
    &= \bigcup_{(j;k) \in \occ(\tau,i)} \Tt(i,j;k)
        \qquad \text{by Proposition~\ref{prop:GIFS}} \displaybreak[0]\\
    &= \bigcup_{\substack{(j;k) \in \occ(\tau,i) \\ j \notin \{a,b\}}} \Tt(i,j;k)
    \ \cup \ \bigcup_{\substack{(a;k) \in \occ(\tau,i)}} \Tt(i,a;k)
    \ \cup \ \bigcup_{\substack{(b;k) \in \occ(\tau,i) }} \Tt(i,b;k) \displaybreak[0]\\
    &= \bigcup_{\substack{(j;k) \in \occ(\tau,i) \\ j \notin \{a,b\}}} \Tt(i,j;k)
    \ \cup \ \bigcup_{\substack{(a;k) \in \occ(\tau,i) }} \Tt(i,a;k) \cup \Tt(i,b;k) \displaybreak[0]\\
    &= \bigcup_{\substack{(j;k) \in \occ(\tau,i) \\ j \notin \{a,b\}}} \Ts(i',j;k)
    \ \cup \ \bigcup_{\substack{(a;k) \in \occ(\tau,i) }} \Ts(i',a;k)
        \qquad \text{by Lemma~\ref{lemm:split_lemm2}} \displaybreak[0]\\
    &=   \bigcup_{\substack{(j;k) \in \occ(\tau,i) \\ j \neq b}} \Ts(i',j;k).
\end{align*}
The third line of the above equation follows from the second line
because $(a;k) \in \occ(\tau,i)$ if and only if $(b;k) \in \occ(\tau,i)$,
by definition of symbol splittings.
Statements~\ref{item:split1},~\ref{item:split2}, \ref{item:split3} of Proposition~\ref{prop:split}
can now be proved by combining the above equality
and the fact that the condition
``$(j;k) \in \occ(\tau,i)$ and $j \neq b$''
is equivalent to
\begin{itemize}
\item
$(j;k) \in \occ(\sigma,i)$ if $i \notin \{a,b\}$, which proves~\ref{item:split1};
\item
$(j;k) \in \occ(\sigma,a) \setminus I$ if $i = a$, which proves~\ref{item:split2};
\item
$(j;k) \in I$ if $i = b$, which proves~\ref{item:split3}.
\end{itemize}
Note that Statement~\ref{item:split1} of Proposition~\ref{prop:split}
was already established in the proof of Lemma~\ref{lemm:split_lemm2}.
\end{proof}

\section{Conjugation by free group automorphisms}
\label{sect:fgaut}

In this section we describe the action of a particular family of free group automorphisms
on the Rauzy fractal of a substitution in Proposition~\ref{prop:fgaut},
which will be used to prove our main result, Theorem~\ref{theo:countablefg}.

A \tdef{free group morphism} on the alphabet $\mcA$
is a non-erasing morphism of the free group generated by $\mcA$,
consisting of the finite words made of symbols $a$ and $a^{-1}$ for $a \in \mcA$.
Substitutions can be seen as a particular case of free group automorphisms,
where no ``$\phantom{}^{-1}$'' appears in the image of each letter.
The \tdef{inverse} of a free group automorphism $\rho$
is the unique morphism (denoted by $\rho^{-1}$) such that
$\rho \rho^{-1} = \rho^{-1} \rho$ is the identity.
For example, the inverse of $\rho : 1 \mapsto 1, 2 \mapsto 211$,
is $\rho^{-1} : 1 \mapsto 1, 2 \mapsto 21^{-1}1^{-1}$.

The fundamental operation we will perform on a substitution $\sigma$ is
\tdef{conjugation by a free group automorphism $\rho$},
\emph{i.e.}, forming the product $\rho^{-1} \sigma \rho$ where $\rho$ is an automorphism.
In the specific cases that we will consider,
$\sigma$ and $\rho^{-1} \sigma \rho$ will always both be substitutions (\emph{i.e.}, contain no ``$\phantom{}^{-1}$'').
We will use a particular family of free group automorphisms,
consisting of the mappings~$\rho_{ij}$ given by (together with their inverses)
\[
\rho_{ij}(k) =
\begin{cases}
ij & \text{ if } k = j \\
k & \text{ if } k \neq j,
\end{cases}
\qquad
\rho_{ij}^{-1}(k) =
\begin{cases}
i^{-1}j & \text{ if } k = j \\
k & \text{ if } k \neq j.
\end{cases}
\]
The next proposition describes how a Rauzy fractal
is affected when its associated substitution is conjugated by a free group automorphism $\rho_{ij}$.
Example~\ref{exam:tribo_fgaut} and Example~\ref{exam:tribo_fgaut2}
provide examples of conjugacy of free groups,
and their actions on Rauzy fractals when combined with symbol splittings.

\begin{prop}
\label{prop:fgaut}
Let
\begin{itemize}
\item
$\tau$ be a primitive unimodular Pisot substitution on the alphabet $\mcA$,
\item
$b \in \mcA$ be such that there exists a unique $c \in \mcA$
such that for every $(j;k) \in \occ(\tau,b)$,
we have $k \geq 2$ and $\tau(j)_{k-1} = c$,
\item
$\theta = \rho_{cb}^{-1} \tau \rho_{cb}^{\phantom{-}}$,
\item
$\wbeta \in \bbR^{n+1}$ be a left eigenvector of $\Mt$ associated with $\beta$,
\item
$\zbeta = \wbeta\bfM_{\rho_{cb}} \in \bbR^{n+1}$
    (which is a left eigenvector of $\Mth$ associated with $\beta$),
\item
$\Tt$ be the Rauzy fractal of $\tau$ (with respect to eigenvector $\wbeta$),
\item
$\Tth$ be the Rauzy fractal of $\theta$ (with respect to eigenvector $\zbeta$).
\end{itemize}
We have
\begin{enumerate}
\item
\label{item:fgaut1}
$\Tth(i) = \Tt(i)$ if $i \notin \{b,c\}$,
\item
\label{item:fgaut2}
$\Tth(b) \cup \Tth(c) = \Tt(c)$.
\end{enumerate}
More precisely,
\begin{enumerate}[start=3]
\item
\label{item:fgaut3}
$\Tth(b) = \displaystyle\bigcup_{(j;k) \in \occ(\tau,b)} \Tt(c,j;k-1)=\Tt(b)-\pi_\tau \bfP(c)$,
\item
\label{item:fgaut4}
$\Tth(c) = \displaystyle\bigcup_{\substack{(j;k) \in \occ(\tau,c) \\ (j;k+1) \notin \occ(\tau,b)}} \Tt(c,j;k).$
\end{enumerate}
In particular, $\Tth = \bigcup_{i \neq b} \Tt(i)$.
\end{prop}

\begin{proof}
We first check that $\zbeta$ is a left eigenvector of $\Mth$ associated with $\beta$:
\[
\zbeta\Mth=\wbeta \bfM_{\rho_{cb}} \Mth
    = \wbeta \bfM_{\rho_{cb}} \bfM_{\rho_{cb}}^{-1} \Mt \bfM_{\rho_{cb}}
    = \wbeta \Mt \bfM_{\rho_{cb}}
    = \beta \wbeta \bfM_{\rho_{cb}}=\beta\zbeta.
\]
Hence, the Rauzy fractal $\Tth$ mentioned in the statement of the Proposition is well-defined.
We now prove Statements~\ref{item:fgaut1} and~\ref{item:fgaut2}.
Let $u$ be a periodic point of $\theta$
(\emph{i.e.}, there exists $k \geq 1$ such that $\theta^k(u) = u$),
and let $v = \rho_{cb}(u)$.
It is easy to check that $v$ is a periodic point of $\tau$:
\[
\tau^k(\rho_{cb}^{\phantom{-}}(u))
    = (\rho_{cb}^{\phantom{-}} \theta \rho_{cb}^{-1})^k\rho_{cb}^{\phantom{-}}(u)
    = \rho_{cb}^{\phantom{-}} \theta^k \rho_{cb}^{-1}\rho_{cb}^{\phantom{-}}(u)
    = \rho_{cb}^{\phantom{-}}\theta^k(u)
    = \rho_{cb}^{\phantom{-}}(u).
\]
Let $\ell : \bbN \rightarrow \bbN$ be the unique function defined by induction as follows for $m \geq 2$:
\[
\ell(1) =
\begin{cases}
1 \textrm{ if } u_1 \neq b \\
2 \textrm{ if } u_1 = b,
\end{cases}
\qquad
\ell(m) =
\begin{cases}
\ell(m-1) + 1 \textrm{ if } u_m \neq b \\
\ell(m-1) + 2 \textrm{ if } u_m = b.
\end{cases}
\]
In particular, we have $u_m = v_{\ell(m)}$ for all $m \in \bbN$.
The definition of $\ell$ is illustrated below, on an example
where $u_3 = v_4 = b$, $v_3 = c$ and $u_1,u_2,u_4,v_1,v_2,v_5 \notin \{b,c\}$.
\[
\begin{array}{rcccccccc}
               u & = & u_1 & u_2 &        & u_3=b & u_4 & \cdots \\
\rho_{cb}(u) = v & = & v_1 & v_2 & v_3 =c & v_4=b & v_5 & \cdots \\
                 & = & v_{\ell(1)} & v_{\ell(2)} & v_{\ell(2)+1} & v_{\ell(3)} & v_{\ell(4)} & \cdots
\end{array}
\]

The equality $\zbeta = \wbeta \bfM_{\rho_{cb}}$
implies that $\langle \zbeta, \bfe_i \rangle = \langle \wbeta, \bfe_i \rangle$ if $i \neq b$
and $\langle \zbeta, \bfe_b \rangle = \langle \wbeta, \bfe_c \rangle + \langle \wbeta, \bfe_b \rangle$.
Hence, by definition of $\pi_\tau$ and $\pi_\theta$
we have $\pi_\theta\bfP(i) = \pi_\tau\bfP(i)$ if $i \neq b$
and $\pi_\theta\bfP(b) = \pi_\tau\bfP(c) + \pi_\tau\bfP(b)$.
It follows that for all $m \in \bbN$,
\begin{align*}
\pi_\theta \bfP(u_1 \cdots u_m)
    &= \sum_{\substack{1 \leq k \leq m \\ u_k \neq b}} \pi_\theta\bfP(u_k)
     + \sum_{\substack{1 \leq k \leq m \\ u_k = b}}  \pi_\theta\bfP(b) \\
    &= \sum_{\substack{1 \leq k \leq m \\ u_k \neq b}} \pi_\tau\bfP(v_{\ell(k)})
     + \sum_{\substack{1 \leq k \leq m \\ u_k = b}} (\pi_\tau\bfP(c) + \pi_\tau\bfP(b)) \\
    &= \sum_{\substack{1 \leq k \leq m \\ u_k \neq b}} \pi_\tau\bfP(v_{\ell(k)})
     + \sum_{\substack{1 \leq k \leq m \\ u_k = b}} \pi_\tau\bfP(v_{\ell(k)-1}v_{\ell(k)}) \\
    &=\pi_\tau \bfP(v_{\ell(1)} \cdots v_{\ell(m)}).
\end{align*}
It is easy to verify that for all $i\notin\{b,c\}$,
$\ell$ is a bijection between $\{m\in\bbN:u_{m+1}=i\}$ and $\{m\in\bbN:v_{m+1}=i\}$,
so
\begin{align*}
\Tth(i) &=\overline{ \{\pi_\theta\bfP(u_1 \cdots u_m) : u_{m+1} = i\}} \\
        &=\overline{ \{\pi_\tau\bfP(v_1 \cdots v_{\ell(m)}) : v_{\ell(m)+1} = i\}} \\
        &= \Tt(i).
\end{align*}
Moreover, $\ell$ is also a bijection between $\{m\in\bbN:u_{m+1}=b\textrm{ or }c\}$ and $\{m\in\bbN:v_{m+1}=c\}$,
so
\begin{align*}
\Tth(b) \cup \Tth(c)
       &=\overline{ \{\pi_\theta\bfP(u_1 \cdots u_m) : u_{m+1} = b \text{ or } c\}}\\
       &= \overline{\{\pi_\tau\bfP(v_1 \cdots v_{\ell(m)}) : v_{\ell(m)+1} = c\}} \\
       &= \Tt(c).
\end{align*}
Therefore, Statements~\ref{item:fgaut1} and~\ref{item:fgaut2} are proved.
For the second equality of Statement~\ref{item:fgaut3}, we compute
\begin{align*}
\Tth(b) +\pi_\theta \bfP(c)
       &=\overline{ \{\pi_\theta\bfP(u_1 \cdots u_m)+\pi_\theta \bfP (c) : u_{m+1} = b \}}\\
       &= \overline{\{\pi_\tau\bfP(v_1 \cdots v_{\ell(m)}c) : v_{\ell(m)+2} = b\}} \\
       &= \Tt(b).
\end{align*}
We used here that $m\mapsto \ell(m)+1$ is a bijection between the sets $\{m\in\bbN:u_{m+1}=b\}$ and $\{m\in\bbN:v_{m+1}=b\}$.
We now prove the first equality of Statement~\ref{item:fgaut3}
(giving a precise description of the subsubtile decomposition of $\Tth(b)$).
Note that
\[
\theta(i) =
\begin{cases}
\text{the word $\tau(i)$ in which each occurrence of $cb$ is replaced by $b$, if $i\ne b$} \\
\text{the word $\tau(cb)=\tau(c)\tau(b)$ in which each occurrence of $cb$ is replaced by $b$}, \text{ if } i = b.
\end{cases}
\]
We apply Proposition~\ref{prop:GIFS}, and use the fact that $\bfh_\theta=\bfh_\tau$ as well as the above correspondence between the occurrences of $b$ in $\theta$ and in $\tau$:
\begin{align}
\Tth(b) &= \bigcup_{(j;k)\in\occ(\theta,b)}
    \bfh_\theta\Tth(j)+\pi_\theta\bfP(\theta(j)_1\cdots\theta(j)_{k-1}) \nonumber \\
&= \bigcup_{(j;k')\in\occ(\tau,b),j\ne b,c}
    \bfh_\tau\Tt(j)+\pi_\tau\bfP(\tau(j)_1\cdots\tau(j)_{k'-2}) \nonumber \\
&\phantom{=}~ \cup\bigcup_{(c;k')\in\occ(\tau,b)}
    \bfh_\theta\Tth(c)+\pi_\tau\bfP(\tau(c)_1\cdots\tau(c)_{k'-2})\label{eq1Tb}\tag{$*$} \\
&\phantom{=}~ \cup\bigcup_{(c;k')\in\occ(\tau,b)}
    \bfh_\theta\Tth(b)+\pi_\tau\bfP(\tau(c)_1\cdots\tau(c)_{k'-2})\label{eq2Tb}\tag{$*${}$*$} \\
&\phantom{=}~ \cup\bigcup_{(b;k')\in\occ(\tau,b)}
    \bfh_\theta\Tth(b)+\pi_\tau\bfP(\tau(c)\tau(b)_1\cdots\tau(b)_{k'-2}).\label{eq3Tb}\tag{$*${}$*${}$*$}
\end{align}

Recall that $\Tth(b) \cup \Tth(c)=\Tt(c)$,
hence, $\bfh_\theta(\Tth(b) \cup \Tth(c))=\bfh_\tau\Tt(c)$,
which allows us to combine \eqref{eq1Tb} and \eqref{eq2Tb} into a single union.
Since
$\pi_\tau\bfP(\tau(c))=\bfh_\tau\pi_\tau\bfP(c)=\bfh_\theta\pi_\theta\bfP(c)$
and $\Tth(b) +\pi_\theta \bfP(c)= \Tt(b)$,
we can write in \eqref{eq3Tb} that
\[
\bfh_\theta\Tth(b)+\pi_\tau\bfP(\tau(c)\tau(b)_1\cdots\tau(b)_{k'-2})=\bfh_\tau\Tt(b)+\pi_\tau\bfP(\tau(b)_1\cdots\tau(b)_{k'-2}).
\]
Therefore, Statement~\ref{item:fgaut3} follows from
\[
\Tth(b) = \bigcup_{(j;k')\in\occ(\tau,b)}\bfh_\tau\Tt(j)+\pi_\tau\bfP(\tau(j)_1\cdots\tau(j)_{k'-2})
=\bigcup_{(j;k') \in \occ(\tau,b)} \Tt(c,j;k'-1).
\]
Statement~\ref{item:fgaut4} can be proved via a similar computation for $\Tth(c)$.
\end{proof}

\begin{exam}
\label{exam:tribo_fgaut}
Let $\sigma : 1 \mapsto 21, 2 \mapsto 31, 3 \mapsto 1$.
First we split $\sigma^3 : 1 \mapsto 1213121, 2 \mapsto 213121, 3 \mapsto 3121$
from $a = 1$ to $b = 4$ with occurrences $I = \{(1;7)\}$
to obtain $\tau : 1 \mapsto 1213124, 2 \mapsto 213121, 3 \mapsto 3121, 4 \mapsto 1213124$.
Then we conjugate $\tau$ with $\rho_{24} : 1 \mapsto 1, 2 \mapsto 2, 3 \mapsto 3, 4 \mapsto 24$:
\[
\rho_{24}^{-1} \tau \rho_{24}^{\phantom{-}}:
\left\{
\begin{array}{lclclcl}
1 & \mapsto & 1   & \mapsto & 1213124       & \mapsto   & 121314      \\
2 & \mapsto & 2   & \mapsto & 213121        & \mapsto   & 213121      \\
3 & \mapsto & 3   & \mapsto & 3121          & \mapsto   & 3121        \\
4 & \mapsto & 24  & \mapsto & 2131211213124 & \mapsto   & 213121121314
\end{array}
\right.
\]
The effect of these operations on the Rauzy fractals are shown in Figure~\ref{fig:strategy}:
the subtiles of $\Ts$ are shown in \subref{fig:strategy-a},
the subsubtiles of $\mcT_{\sigma^3}$ are shown in \subref{fig:strategy-b},
the subtiles of $\Tt$ are shown in \subref{fig:strategy-c}
and the subtiles of $\Tth$ are shown in \subref{fig:strategy-d}.
\end{exam}

\begin{exam}
\label{exam:tribo_fgaut2}
Let $\sigma : 1 \mapsto 21, 2 \mapsto 31, 3 \mapsto 1$.
Let $\tau$ be the splitting of $\sigma^6$ from $1$ to $4$ with occurrences
$I = \{(1;24);(1;31);(1;33);(1;40)\}$.
(Note that $\sigma(1)_{p-1} = 2$ for all $(1,p) \in I$.)
Let $\theta = \rho_{24}^{-1} \tau \rho_{24}$ with
$\rho_{24} : 1 \mapsto 1, 2 \mapsto 2, 3 \mapsto 3, 4 \mapsto 24$.
The effect of these operations on the Rauzy fractal are shown in Figure~\ref{fig:fgaut2}.
\begin{figure}[ht]
\centering
\includegraphics[height=35mm]{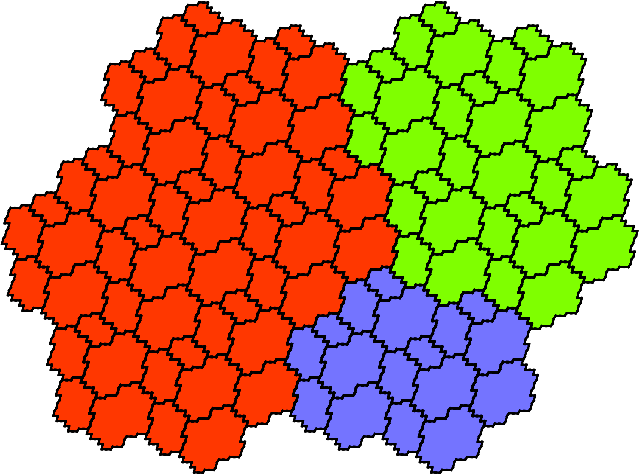} \hfil
\includegraphics[height=35mm]{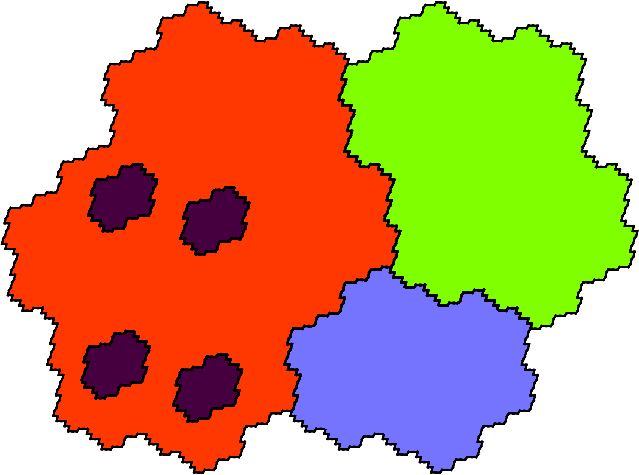} \hfil
\includegraphics[height=35mm]{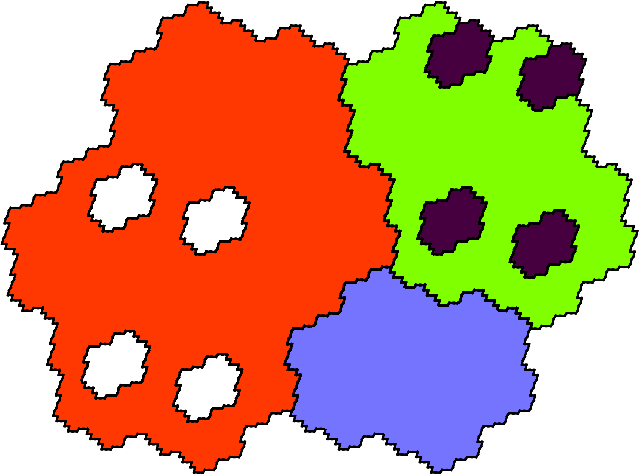} \hfil
\caption{
    Rauzy fractals of the substitutions defined in Example~\ref{exam:tribo_fgaut2}.
    The subsubtiles of $\mcT_{\sigma^6}$ (left),
    the subtiles of $\Tt$ (center),
    and the subtiles of $\Tth$ (right).
}
\label{fig:fgaut2}
\end{figure}
\end{exam}

\section{Main results}
\label{sect:mainresults}

We now combine the results of Section~\ref{sect:split} (symbol splittings)
and Section~\ref{sect:fgaut} (conjugations by free group automorphisms)
in order to prove our main result, Theorem~\ref{theo:countablefg}.
First, we prove in Proposition~\ref{prop:itersubtiles}
that it is possible to ``dig holes'' in a ``nice'' planar Rauzy fractal by extracting some subsubtiles.

\begin{prop}
\label{prop:itersubtiles}
Let $\sigma$ be a primitive unimodular Pisot substitution on the alphabet $\mcA$
with dominant Pisot eigenvalue of degree $3$,
such that $\Ts$ and its subtiles $\Ts(i)$ ($i\in\mcA$) are homeomorphic to a closed disc.
Let $K\geq 1$.
Then there exist two letters $a\ne c\in\mcA$ and $N \geq 1$ and
\[
I \subset \underbrace{\{(j;k) \in \occ(\sigma^N,a) : \sigma^N(j)_{k-1} = c\}}_{\displaystyle=:O_N}\subset\occ(\sigma^N,a)
\]
such that the sets
\[
\overline{\Ts(a) \setminus \bigcup_{(j;k) \in I} \mcT_{\sigma^N}(a,j;k)}
    = \bigcup_{(j;k) \in \occ(\sigma^N,a)\setminus I} \mcT_{\sigma^N}(a,j;k)
\]
and
\[
\overline{\Ts \setminus \bigcup_{(j;k) \in  I} \mcT_{\sigma^N}(a,j;k)}
\]
are homeomorphic to a closed disc minus the union of $K$ disjoint open discs of its interior.
Here, $\overline{M}$ denotes the closure of a set $M$.
\end{prop}

\begin{proof}
Let $a\ne c \in \mcA$ such that the word $ca$ occurs in a power of $\sigma$, \emph{i.e.},
such that $\sigma^{n_0}(j_0)_{k_0} = a$ and $\sigma^{n_0}(j_0)_{k_0-1} = c$ for some $j_0\in\mcA$
and $n_0,k_0\geq 1$. We will dig holes in the subsubtile $\mcT_{\sigma^{n_0}}(a,j_0;k_0)\subset\Ts(a)$.
Let $x_1, \ldots, x_K$ be $K$ points in the interior of $\mcT_{\sigma^{n_0}}(a,j_0;k_0)$
and $B_1, \ldots, B_K \subset\mcT_{\sigma^{n_0}}(a,j_0;k_0)$ be $K$ disjoint closed discs such that
for each $m\in\{1,\ldots,K\}$, $B_m$ is centered at $x_m$.
We can assume that $B_1, \ldots, B_K$ all have the same radius $r > 0$
and that their boundaries do not intersect the boundary of $\mcT_{\sigma^{n_0}}(a,j_0;k_0) $.

We can choose $N$ such that all the subsubtiles of $\sigma^N$ have diameter less than $r/2$,
because each subsubtile of $\sigma^N$ is a copy of a subsubtile of $\sigma$ which is scaled down by $\bfh_\sigma^N$,
and $\bfh_\sigma$ is a contraction.
Thanks to Proposition~\ref{prop:GIFS} and Remark~\ref{rema:powers}, we can further choose
a set of $K$ occurrences $I := \{(j_1;k_1), \ldots, (j_p;k_p)\}\subset\occ(\sigma^N,a)$
such that for every $m \in \{1, \ldots, K\}$, we have $x_m \in \mcT_{\sigma^N}(a,j_m;k_m)$.
Each subsubtile $\mcT_{\sigma^N}(a,j_m;k_m)$ has diameter less than $r/2$, thus it is contained in $B_m$.
In particular, the subsubtiles  $\mcT_{\sigma^N}(a,j_1;k_1),\ldots, \mcT_{\sigma^N}(a,j_K;k_K)$
are all disjoint and contained in $\mcT_{\sigma^{n_0}}(a,j_0;k_0)$.

We claim that  $I \subset O_N$.  Indeed,  by assumption, there exist a prefix $p\in\mcA^*$
and a suffix $s\in\mcA^*$ satisfying $\sigma^{n_0}(j_0)=pas$, $|p|=k_0-1$ and $p_{k_0-1}=c$.
Moreover, it follows from the inclusion $\mcT_{\sigma^N}(a,j_m;k_m)\subset \mcT_{\sigma^{n_0}}(a,j_0;k_0)$
that there exist $p_m',s_m'\in\mcA^*$ with the property that $\sigma^N(j_m)=\sigma^{n_0}(p_m')\;pas\;\sigma^{n_0}(s_m')$
and $|\sigma^{n_0}(p')p|=k_m-1$. Thus $\sigma^N(j_m)_{k_m-1}=c$ and $(j_m;k_m)\in O_N$, for each $m\in\{1,\ldots,K\}$.

Each subsubtile $\mcT_{\sigma^N}(a,j_m;k_m)$ is a set which is homeomorphic to a disc and which is contained in $B_m$.
Therefore, by Sch\"onflies' theorem~\cite{Thom92},
the closure of $\Ts\setminus \bigcup_{(j;k) \in I} \mcT_{\sigma^N}(a,j;k)$
is homeomorphic to a disc from which $K$ open discs with disjoint boundaries have been removed.
The same property holds for
\[
\overline{\Ts(a) \setminus \bigcup_{(j;k) \in  I} \mcT_{\sigma^N}(a,j;k)}
    = \bigcup_{(j;k) \in \occ(\sigma^N,a)\setminus I} \mcT_{\sigma^N}(a,j;k)
\]
(note that the union on the right side is measure-disjoint by Proposition~\ref{prop:GIFS}).
\end{proof}

We are now able to prove our main result.

\begin{theo}
\label{theo:countablefg}
Let $K \geq 1$ be an integer and denote by $F_K$ the free group of rank $K$.
Then:
\begin{enumerate}
\item
\label{item:cfg_subtile}
There exists a $4$-letter primitive unimodular Pisot substitution $\tau$
such that the fundamental group of a subtile $\Tt(a)$ of the Rauzy fractal $\Tt$ is isomorphic to $F_K$,
and such that the subtiles of $\Tt$ have disjoint interiors.
\item
\label{item:cfg_rf}
There exists a $4$-letter primitive unimodular Pisot substitution $\theta$
such that the fundamental group of the Rauzy fractal $\Tth$ is isomorphic to $F_K$
and such that the subtiles of $\Tth$ have disjoint interiors.
\end{enumerate}
\end{theo}

\begin{proof}
Let $\sigma$ be any primitive unimodular substitution on the alphabet $\{1,2,3\}$
whose dominant eigenvalue is a cubic Pisot number,
and such that $\Ts$ and its subtiles are homeomorphic
to a disc and have disjoint interiors.
We also require $\sigma$ to satisfy the strong coincidence condition (see Remark~\ref{rema:SCC}).
One of the many possible choices for $\sigma$ is the Tribonacci substitution $1 \mapsto 12, 2 \mapsto 13, 3 \mapsto 1$,
for which the above properties can easily be verified~\cite{ST09,CANT}.

Let $a\ne c\in\mcA$, $N \geq 1$ and
\[
I \subset \underbrace{\{(j;k) \in \occ(\sigma^N,a) : \sigma^N(j)_{k-1} = c\}}_{\displaystyle=:O_N}\subset\occ(\sigma^N,a)
\]
as given in Proposition~\ref{prop:itersubtiles}.
Then $\bigcup_{(j;k) \in \occ(\sigma^N,a)\setminus I} \mcT_{\sigma^N}(a,j;k)$
is homeomorphic to a closed disc minus the union of $K$ disjoint open discs.
Thus the fundamental group of this set is isomorphic to $F_K$.
Let $\tau$ be the substitution obtained from $\sigma^N$ by splitting $a$ to a new symbol $b$ with occurrences $I$.
By Proposition~\ref{prop:split} we have
$\Tt(a) = \bigcup_{(j;k) \in \occ(\sigma^N,a)\setminus I} \mcT_{\sigma^N}(a,j;k)$,
so $\Tt(a)$ is isomorphic to $F_K$.

The subtiles of $\Tt$ are measure-disjoint.
Indeed, $\sigma^N$ is a primitive unimodular Pisot substitution,
thus for each $i\in\mcA$, the subsubtiles of $\mcT_{\sigma^N}(i)$ are measure-disjoint
for each $i\in\mcA$ (see~Proposition~\ref{prop:GIFS}).
Moreover, we chose $\sigma$ satisfying the strong coincidence condition,
so the subtiles $\Ts(i)=\mcT_{\sigma^N}(i)$ ($i\in\mcA$) of $\mcT_{\sigma^N}$
are also measure-disjoint (see Remark~\ref{rema:SCC}).
Therefore, Proposition~\ref{prop:split} implies that the subtiles of $\Tt$ are measure-disjoint
and Statement~\ref{item:cfg_subtile} is proved.

To prove Statement~\ref{item:cfg_rf}, we make use of the specific choice $I\subset O_N$ for the set of occurrences.
Indeed, a conjugation will have the effect of moving the subsubtiles associated with $I$
from the tile associated with $b$ into another tile, hence leaving $K$ holes in the fractal.
We apply Proposition~\ref{prop:fgaut} to $\tau$,
which is a primitive unimodular Pisot substitution on the alphabet $\mcA':=\mcA\cup\{b\}$,
and to $\theta= \rho_{cb}^{-1} \tau \rho_{cb}^{\phantom{-}}$.
The assumption on the occurrences of $b$ is fulfilled because $I\subset O_N$.
Remember that Proposition~\ref{prop:split} also asserted that $\Tt(i) = \mcT_{\sigma^N}(i)$ if $i \notin \{a,b\}$.
It follows that
\begin{align*}
\Tth
    &= \bigcup_{i\ne b}\Tt(i)
       \ = \ \bigcup_{i\ne a,b}\Tt(i) \ \cup \ \Tt(a) \\
    &= \underbrace{\bigcup_{i\ne a,b}\mcT_{\sigma^N}(i)}_{\bigcup_{i\ne a}\mcT_{\sigma^N}(i)}
       \ \cup \ \bigcup_{(j;k) \in \occ(\sigma^N,a)\setminus I} \mcT_{\sigma^N}(a,j;k) \\
    &= \overline{\Ts\setminus\bigcup_{(j;k) \in I} \mcT_{\sigma^N}(a,j;k)}
\end{align*}
and its fundamental group is isomorphic to $F_K$ by Proposition~\ref{prop:itersubtiles}.

Finally, as all subtiles $\Tt(i)$ and there subsubtiles are measure-disjoint,
we can infer from Proposition~\ref{prop:fgaut} that the subtiles of $\Tth$ are also measure-disjoint.
\end{proof}

\section{Conclusion}

Our results are obtained with a fixed number of symbols ($4$)
so there is no bound of the number of holes by the number of symbols,
which answers a question asked to the authors by Minervino.
It is not known whether there exists a $3$-letter Pisot substitution
with nontrivial but countable fundamental~group.

In further developments, we may try to realize higher homology/homotopy groups
for three-dimensional Rauzy fractals associated with Pisot numbers of degree $\geq 4$.
Indeed, illustrations of Figure~\ref{fig:3D_holes} lead to think that our methods could be adapted to higher dimensions,
since Propositions~\ref{prop:split} and~\ref{prop:fgaut} do not assume planarity of the tiles.
However this is out of reach for the moment,
because we need the essential preliminary fact that the subtiles are homeomorphic to a ball,
but appropriate criteria in $3$-dimensions do not currently exist.
A reason is that the theorem of Sch\"onflies used in Proposition~\ref{prop:itersubtiles} does not generalize to higher dimensions.
Developments in this direction have recently been obtained by Conner and Thuswaldner~\cite{CT14}.

\begin{figure}[ht]
\centering
\includegraphics[height=35mm]{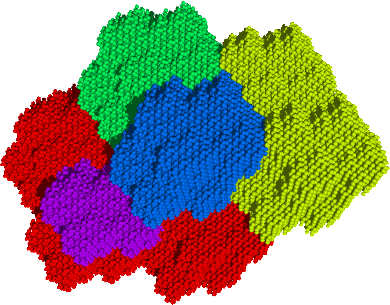}
\hfil
\includegraphics[height=35mm]{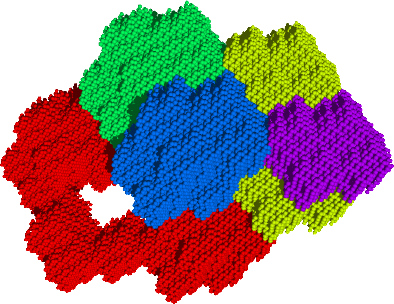}
\caption{
Drilling holes in the ``Quadribonacci'' substitution
$\sigma : 1 \mapsto 21, 2 \mapsto 31, 3 \mapsto 41, 4 \mapsto 1$.
We have splitted the occurrences $\{(1,8), (2,7)\}$
of $1$ to a new symbol $5$ in $\sigma^3$,
and conjugated the result by $\rho_{2,5}$
to obtain the substitution
$\theta : 1 \mapsto 4121315, 2 \mapsto 121315, 3 \mapsto 213121, 4 \mapsto 3121, 5 \mapsto 1213154121315$.
The Rauzy fractals of $\sigma$ and $\theta$ are plotted above left and right, respectively.
These fractals are three-dimensional because the associated Pisot eigenvalue is of degree $4$.
}
\label{fig:3D_holes}
\end{figure}

Another perspective for further work is to describe some \emph{uncountable}
fundamental groups for some simple examples,
such as the fractal shown in Figure~\ref{fig:fg_cool}.
This has successfully been done for some fractals such as the
Hawaiian earring or the Sierpi\'nski triangle~\cite{CC00,ADTW09}.

\begin{figure}[ht]
\centering
\includegraphics[height=38mm]{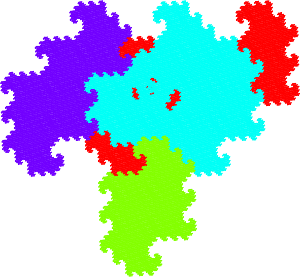}
\hfil
\includegraphics[height=38mm]{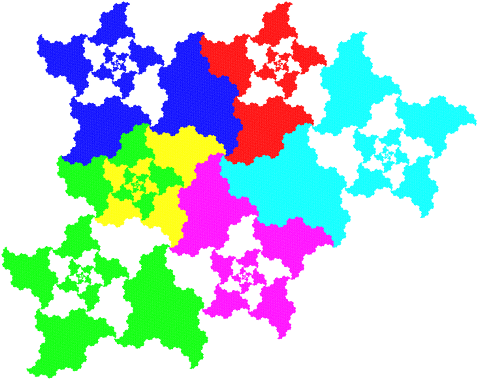}
\caption{
The Rauzy fractals of $1\mapsto2413, 2\mapsto43, 3\mapsto2433, 4\mapsto1$ (left)
and $1\mapsto 2, 2\mapsto 4,3, 3\mapsto 4, 4\mapsto 53, 5\mapsto 6, 6\mapsto 1$ (right).
The first picture suggests that one of the subtiles is homeomorphic
to a disc from which infinitely discs have been removed,
which would make it homeomorphic to the Hawaiian earring.
}
\label{fig:fg_cool}
\end{figure}

\paragraph{Acknowledgements}
We thank the anonymous referees for many suggestions improving the quality of the paper.

The first and the second authors were supported by Agence Nationale de la Recherche and Austrian Science Fund
through the ANR/FWF project FAN, \emph{Fractals and Numeration}, ANR-12-IS01-0002, FWF I1136.
The second author was also supported by the FWF Project 22 855, \emph{Topology of fractal tiles},
of the Austrian Science Fund.
The third author was supported by the Chinese National Natural Science Foundation Project 10971233.

\bibliographystyle{amsalpha}
\bibliography{biblio}

\end{document}